\documentclass[11pt]{amsart}
\usepackage{amsmath, amssymb}
\usepackage[arrow,matrix,curve,cmtip,ps]{xy}

\usepackage{amsthm}

\allowdisplaybreaks

\newtheorem{theorem}{Theorem}[section]
\newtheorem{lemma}[theorem]{Lemma}
\newtheorem{proposition}[theorem]{Proposition}
\newtheorem{corollary}[theorem]{Corollary}

\newtheorem*{theorem*}{Theorem}
\theoremstyle{remark}
\newtheorem{remark}[theorem]{Remark}
\newtheorem{definition}[theorem]{Definition}
\newtheorem{example}[theorem]{Example}

\newcommand{\clspan}{\operatorname{\overline{\textnormal{span}}}}
\newcommand{\reg}{\textnormal{reg}}
\newcommand{\path}{\operatorname{Path}}

\numberwithin{equation}{section}

\newcommand{\N}{\mathbb{N}}

\newcommand{\C}{\mathbb{C}}
\newcommand{\T}{\mathbb{T}}
\newcommand{\Hi}{\mathcal{H}}

\newcommand{\gae}{\lower 2pt \hbox{$\, \buildrel {\scriptstyle >}\over {\scriptstyle
\sim}\,$}}

\newcommand{\lae}{\lower 2pt \hbox{$\, \buildrel {\scriptstyle <}\over {\scriptstyle
\sim}\,$}}

\newcommand{\MU}[1]{
\setbox0\hbox{$#1$}
\setbox1\hbox{$W$}
\ifdim\wd0>\wd1 #1^{\sim} \else \widetilde{#1} \fi
}

\begin{document}
\title[A class of $C^*$-algebras that are prime but not primitive]{A class of $\boldsymbol{C^*}$-algebras that are prime but not primitive}

\author{Gene Abrams}

\author{Mark Tomforde}

\address{Department of Mathematics \\ University of Colorado \\ Colorado Springs, CO 80918 \\USA}
\email{abrams@math.uccs.edu}

\address{Department of Mathematics \\ University of Houston \\ Houston, TX 77204-3008 \\USA}
\email{tomforde@math.uh.edu}

\thanks{This work was supported by Collaboration Grants from the Simons Foundation to each author (Simons Foundation Grant \#20894 to Gene Abrams and Simons Foundation Grant \#210035 to Mark Tomforde)}

\date{\today}



\begin{abstract}
We establish necessary and sufficient conditions on a (not necessarily countable) graph $E$ for the graph $C^*$-algebra $C^*(E)$ to be primitive.  Along with a known characterization of the graphs $E$ for which $C^*(E)$ is prime, our main result provides us with a systematic method for easily producing large classes of (necessarily nonseparable) $C^*$-algebras that are prime but not primitive.   We also compare and contrast our results with similar results for Leavitt path algebras.
\end{abstract}

\maketitle

\section{Introduction}

It is well known that any primitive $C^*$-algebra must be a prime $C^*$-algebra, and a partial converse was established by Dixmier in the late 1950's when he showed that every separable prime $C^*$-algebra is primitive (see \cite[Corollaire~1]{DixmierPaper}  or \cite[Theorem~A.49]{RW} for a proof).  For over 40 years after Dixmier's result, it was an open question as to whether every prime $C^*$-algebra is primitive.  This was answered negatively in 2001 by Weaver, who produced the first example of a (necessarily nonseparable) $C^*$-algebra that is prime but not primitive \cite{W}.  Additional ad hoc examples of $C^*$-algebras that are prime but not primitive have been given in \cite{C}, \cite[Proposition~31]{K},  and  \cite[Proposition~13.4]{K3}, with this last example being constructed as a graph $C^*$-algebra.   In this paper we identify necessary and sufficient conditions on the graph $E$ for the $C^*$-algebra $C^*(E)$ to be primitive.  Consequently, this will provide  a systematic way for easily describing large classes of (necessarily nonseparable) $C^*$-algebras that are prime but not primitive.  In particular, we obtain infinite classes of (nonseparable) AF-algebras, as well as infinite classes of non-AF, real rank zero $C^*$-algebras, that are prime but not primitive.

Compellingly, but perhaps not surprisingly, the conditions on $E$ for which $C^*(E)$ is primitive are identical to the conditions on $E$ for which the Leavitt path algebra $L_K(E)$ is primitive for any field $K$ \cite[Theorem~5.7]{ABR}.   However, as is typical in this context, despite the similarity of the statements of the results, the proofs for graph $C^*$-algebras are dramatically different from the proofs for Leavitt path algebras, and neither  result directly implies the other.

\smallskip

\smallskip

\noindent {\bf Acknowledgement:}   We thank Takeshi Katsura for providing us with useful comments and some helpful observations.

\section{Preliminaries on graph $C^*$-algebras}

In this section we establish notation and recall some standard definitions.

\begin{definition}
A \emph{graph} $(E^0, E^1, r, s)$ consists of a set $E^0$ of vertices, a set $E^1$ of edges, and maps $r : E^1 \to E^0$ and $s : E^1 \to E^0$ identifying the range and source of each edge.  
\end{definition}

\begin{definition}
Let $E := (E^0, E^1, r, s)$ be a graph. We say that a vertex $v
\in E^0$ is a \emph{sink} if $s^{-1}(v) = \emptyset$, and we say
that a vertex $v \in E^0$ is an \emph{infinite emitter} if
$|s^{-1}(v)| = \infty$.  A \emph{singular vertex} is a vertex that
is either a sink or an infinite emitter, and we denote the set of
singular vertices by $E^0_\textnormal{sing}$.  We also let
$E^0_\textnormal{reg} := E^0 \setminus E^0_\textnormal{sing}$, and
refer to the elements of $E^0_\textnormal{reg}$ as \emph{regular
vertices}; i.e., a vertex $v \in E^0$ is a regular vertex if and
only if $0 < |s^{-1}(v)| < \infty$.  A graph is \emph{row-finite}
if it has no infinite emitters.  A graph is \emph{finite} if both
sets $E^0$ and $E^1$ are finite.  A graph is \emph{countable} if both
sets $E^0$ and $E^1$ are (at most) countable.
\end{definition}

\begin{definition}
If $E$ is a graph, a \emph{path} is a finite sequence $\alpha := e_1 e_2
\ldots e_n$ of edges with $r(e_i) = s(e_{i+1})$ for $1 \leq i \leq
n-1$.  We say the path $\alpha$ has \emph{length} $| \alpha| :=n$,
and we let $E^n$ denote the set of paths of length $n$.  We
consider the vertices of $E$ (i.e., the elements of $E^0$) to be paths of length zero.  We
also let $\path (E) := \bigcup_{n \in \N \cup \{0\}} E^n$ denote the set of paths in $E$, and we extend the maps $r$ and $s$ to $\path (E)$
as follows: for $\alpha = e_1 e_2 \ldots e_n \in E^n$ with $n\geq
1$, we set $r(\alpha) = r(e_n)$ and $s(\alpha) = s(e_1)$; for
$\alpha = v \in E^0$, we set $r(v) = v = s(v)$.  Also, for $\alpha = e_1e_2\cdots e_n\in \path (E)$,  we let $\alpha^0$ denote the set of vertices that appear in $\alpha$; that is, $$\alpha^0  = \{ s(e_1), r(e_1), \ldots, r(e_n) \}.$$
\end{definition}

\begin{definition} \label{graph-C*-def}
If $E$ is a graph, the \emph{graph $C^*$-algebra} $C^*(E)$ is the universal
$C^*$-algebra generated by mutually orthogonal projections $\{ p_v
: v \in E^0 \}$ and partial isometries with mutually orthogonal
ranges $\{ s_e : e \in E^1 \}$ satisfying
\begin{enumerate}
\item $s_e^* s_e = p_{r(e)}$ \quad  for all $e \in E^1$
\item $s_es_e^* \leq p_{s(e)}$ \quad for all $e \in E^1$
\item $p_v = \sum_{\{ e \in E^1 : s(e) = v \}} s_es_e^* $ \quad for all $v \in E^0_\textnormal{reg}$.
\end{enumerate}
\end{definition}

\begin{definition} \label{s-alpha-def}
We call Conditions (1)--(3) in Definition~\ref{graph-C*-def} the \emph{Cuntz-Krieger relations}.  Any collection $\{ S_e, P_v : e \in E^1, v \in E^0 \}$ of elements of a $C^*$-algebra $A$,  where the $P_v$ are mutually orthogonal projections, the $S_e$ are partial isometries with mutually orthogonal ranges, and the Cuntz-Krieger relations are satisfied is called a \emph{Cuntz-Krieger $E$-family} in $A$.   The universal property of $C^*(E)$ says precisely that if $\{ S_e, P_v : e \in E^1, v \in E^0 \}$ is a Cuntz-Krieger $E$-family in a $C^*$-algebra $A$, the there exists a $*$-homomorphism $\phi : C^*(E) \to A$ with $\phi (p_v) = P_v$ for all $v \in E^0$ and $\phi(s_e)=S_e$ for all $e \in E^1$.

For a path $\alpha := e_1 \ldots e_n$, we define $S_\alpha := S_{e_1} \cdots S_{e_n}$; and when $|\alpha| = 0$, we have $\alpha = v$ is a vertex and define $S_\alpha := P_v$.
\end{definition}

\begin{remark}\label{productsofsalphasalphastar}
Using the orthogonality of the projections  $\{s_es_e^* : e\in E^1\}$, we see that if $\alpha, \beta \in {\rm Path}(E)$, then $s_\alpha s_\alpha^* s_\beta s_\beta^*$ is nonzero if and only if $\alpha = \beta \gamma$ or $\beta = \alpha \delta$ for some $\gamma, \delta \in {\rm Path}(E)$; in the former case we get $s_\alpha s_\alpha^* s_\beta s_\beta^* = s_\alpha s_\alpha^*$, while in the latter we get $s_\alpha s_\alpha^* s_\beta s_\beta^* = s_\beta s_\beta^*$.   Specifically, if $\mid \alpha \mid = \mid \beta \mid$, then $s_\alpha s_\alpha^* s_\beta s_\beta^*$ is nonzero precisely when $\alpha = \beta$, in which case the product yields $s_\alpha s_\alpha^*$. 
\end{remark}

\begin{definition}
A \emph{cycle} is a path $\alpha=e_1 e_2 \ldots e_n$ with length $|\alpha| \geq 1$ and $r(\alpha) = s(\alpha)$.  If $\alpha = e_1e_2 \ldots e_n$ is a cycle, an \emph{exit} for $\alpha$ is an edge $f \in E^1$ such that $s(f) = s(e_i)$ and $f \neq e_i$ for some $i$.  We say that a graph satisfies \emph{Condition~(L)} if every cycle in the graph has an exit.
\end{definition}

\begin{definition}
A \emph{simple cycle} is a cycle $\alpha=e_1 e_2 \ldots e_n$ with $r(e_i) \neq s(e_1)$ for all $1 \leq i \leq n-1$.  We say that a graph satisfies \emph{Condition~(K)} if no  vertex in the graph is the source of exactly one simple cycle.  (In other words, a graph satisfies Condition~(K) if and only if every vertex in the graph is the source of no simple cycles or the source of at least two simple cycles.)
\end{definition}

Our main use of Condition~(L) will be in applying the Cuntz-Krieger Uniqueness Theorem.  The Cuntz-Krieger Uniqueness Theorem was proven for row-finite graphs in \cite[Theorem~1]{BPRS}, and for countably infinite graphs in \cite[Corollary~2.12]{DT1} and \cite[Theorem~1.5]{RS}.  The result for possibly uncountable graphs is a special case of the Cuntz-Krieger Uniqueness Theorem \cite[Theorem~5.1]{K1} for topological graphs.  Alternatively, one can obtain the result in the uncountable case by using the version for countable graphs and applying the direct limit techniques described in \cite{RS} and \cite{Goo}.

\begin{theorem}[Cuntz-Krieger Uniqueness Theorem] \label{CKUT-thm}
If $E$ is a graph that satisfies Condition~(L) and $\phi : C^*(E) \to A$ is a $*$-homomorphism from $C^*(E)$ into a $C^*$-algebra $A$ with the property that $\phi(p_v) \neq 0$ for all $v \in E^0$, then $\phi$ is injective.
\end{theorem}

It is a consequence of the Cuntz-Krieger Uniqueness Theorem that if $E$ is a graph satisfying Condition~(L) and $I$ is a nonzero ideal of $C^*(E)$, then there exists $v \in E^0$ such that $p_v \in I$.  (To see this, consider the quotient map $q : C^*(E) \to C^*(E)/I$.)  

\begin{definition} \label{sat-hered-def}
If $v, w \in E^0$ we write $v \geq w$ to mean that there exists a path $\alpha \in \path (E)$ with $s(\alpha) = v$ and $r(\alpha) = w$.  (Note that the path $\alpha$ could be a single vertex, so that in particular we have $v \geq v$ for every $v \in E^0$.)

If $E$ is a graph, a subset $H \subseteq E^0$ is \emph{hereditary} if whenever $e \in E^1$ and $s(e) \in H$, then $r(e) \in H$.  Note that a short induction argument shows that whenever  $H$ is hereditary, $v \in H$, and $v \geq w$, then $w \in H$.

For $v\in E^0$ we define
$$H(v) := \{w \in E^0 \ | \ v\geq w\}.$$
It is clear that $H(v)$ is hereditary for each $v\in E^0$.   A hereditary subset $H$ is called \emph{saturated} if $\{ v \in E^0_\reg : r(s^{-1}(v)) \subseteq H \} \subseteq H$.  For any hereditary subset $H$, we let $$\overline{H} := \bigcap \{ K : \text{$K \subseteq H$ and $K$ is a saturated hereditary subset} \}$$ denote the smallest saturated hereditary subset containing $H$, and we call $\overline{H}$ the \emph{saturation} of $H$.  Note that if $H$ is hereditary and we define $H_0 := H$ and $H_n := H_{n-1} \cup \{ v \in E^0_\textnormal{reg} : r(s^{-1}(v)) \subseteq H_{n-1} \}$ for $n \in \N$, then $H_0 \subseteq H_1 \subseteq H_2 \subseteq \ldots$ and $\overline{H} = \bigcup_{n=0}^\infty H_n$.  It is clear that both of the sets $\emptyset$ and $E^0$ are hereditary subsets of $E^0$. 
\end{definition}

\begin{lemma} \label{sat-int-lem}
Let $E = (E^0, E^1, r, s)$ be a graph.  If $H \subseteq E^0$ and $K \subseteq E^0$ are hereditary subsets with $H \cap K = \emptyset$, then $\overline{H} \cap \overline{K} = \emptyset$.  
\end{lemma}

\begin{proof}
Since $H \cap K = \emptyset$, we have $H_0 \cap K_0 = \emptyset$.  A straightforward inductive argument shows that $H_n \cap K_n = \emptyset$ for all $n \in \N$.  Since $H_0 \subseteq H_1 \subseteq H_2 \subseteq  \ldots$ and $K_0 \subseteq K_1 \subseteq K_2 \subseteq  \ldots$, it follows that $\bigcup_{n=0}^\infty H_n \cap \bigcup_{n=0}^\infty K_n = \emptyset$.  Hence $\overline{H} \cap \overline{K} = \emptyset$.
\end{proof}

If $H$ is a hereditary subset, we let 
$$I_H := \clspan \big( \{ s_\alpha s_\beta^* : \alpha,  \beta \in \path(E), \  r(\alpha) = r( \beta) \in H \}).$$
It is straightforward to verify that if $H$ is hereditary, then $I_H$ is a (closed, two-sided) ideal of $C^*(E)$ and $I_H = I_{\overline{H}}$.  In addition, the map $H \mapsto I_H$ is an injective lattice homomorphism from the lattice of saturated hereditary subsets of $E$ into the gauge-invariant ideals of $C^*(E)$.  (In particular, if $H$ and $K$ are saturated hereditary subsets of $E$, then $I_H \cap I_K = I_{H \cap K}$.)  When $E$ is row-finite, the lattice homomorphism $H \mapsto I_H$ is surjective onto the gauge-invariant ideals of $C^*(E)$.

\begin{definition}
A graph $E$ is \emph{downward directed} if for all $u, v \in E^0$, there exists $w \in E^0$ such that $u \geq w$ and $v \geq w$.
\end{definition}

\begin{lemma} \label{reachability-implies-in-I-lem}
Let $E = (E^0, E^1, r, s)$ be a graph, and let $v,w \in E^0$.  If $I$ is an ideal in $C^*(E)$, $p_v \in I$, and $v \geq w$, then $p_w \in I$.
\end{lemma}

\begin{proof}
Let $\alpha$ be a path with $s(\alpha) = v$ and $r(\alpha) =w$.  Since $p_v \in I$, we have

\hspace{1in} $p_w = p_{r(\alpha)} = s_\alpha^* s_\alpha = s_\alpha p_{s(\alpha)} s_\alpha^* = s_\alpha p_v s_\alpha^* \in I.$
\end{proof}

The following  graph-theoretic notion will play a central role in this paper.

\begin{definition}
Let $E = (E^0, E^1, r, s)$ be a graph.  For $w\in E^0$, we define $$U(w) := \{ v \in E^0 : v \geq w \};$$ that is,  $U(w)$ is the set of vertices $v$ for which there exists a path from $v$ to $w$.   We say  $E$ satisfies the {\it Countable Separation Property}  if there exists a countable set $X\subseteq E^0$ for which $E^0 = \bigcup_{x\in X}U(x)$.
\end{definition}

\begin{remark}
It is useful to note that  $E$ does {\it not} satisfy the Countable Separation Property if and only if for every countable subset $X \subseteq E^0$ we have $E^0 \setminus \bigcup_{x\in X}U(x) \neq \emptyset$.   
\end{remark}

\subsection{The notions of ``prime" and ``primitive" for algebras, and for $C^*$-algebras}

When we are working with a ring $R$, an \emph{ideal} in $R$ shall always mean a two-sided ideal.  When working with a $C^*$-algebra $A$, an \emph{ideal} in $A$ shall always mean a closed two-sided ideal.  If we have a two-sided ideal in a $C^*$-algebra that is not closed, we shall refer to it as an \emph{algebraic ideal}.  

Many properties for rings are stated in terms of two-sided ideals.  However, when working in the category of $C^*$-algebras it is natural to consider the corresponding $C^*$-algebraic properties stated in terms of closed two-sided ideals.  Thus for $C^*$-algebras, one may ask whether a given ring-theoretic property coincides with the corresponding $C^*$-algebraic property.  In the next few definitions we will consider the notions of prime and primitive, and explain how the ring versions of these properties coincide with the $C^*$-algebraic versions.  In particular, a $C^*$-algebra is prime as a $C^*$-algebra if and only if it is prime as a ring, and a $C^*$-algebra is primitive as a $C^*$-algebra if and only if it is primitive as a ring.  This will allow us to unambiguously refer to $C^*$-algebras as ``prime" or ``primitive".

If $I,J$ are two-sided ideals of a ring $R$, then the product $IJ$ is defined to be the two-sided ideal
$$IJ := \left\{ \sum_{\ell=1}^n i_\ell j_\ell  \ | \ n\in \N, i_\ell \in I, j_\ell  \in J \right\} .$$
\begin{definition} \label{prime-ring-def}
A ring $R$ is {\it prime} if whenever $I$ and $J$ are two-sided ideals of $R$ and $IJ = \{ 0 \}$, then either $I = \{ 0 \}$ or $J = \{ 0 \}$. 
\end{definition}

If $I,J$ are closed two-sided ideals of a $C^*$-algebra $A$, then the product $IJ$ is defined to be the closed two-sided ideal
$$IJ := \overline{IJ} = \overline{ \left\{ \sum_{\ell=1}^n i_\ell j_\ell  \ | \ n\in \N, i_\ell \in I, j_\ell  \in J \right\}} .$$

\begin{definition} \label{prime-C*-def}
A $C^*$-algebra $A$ is {\it prime} if whenever $I$ and $J$ are closed two-sided ideals of $A$ and $IJ = \{ 0 \}$, then either $I = \{ 0 \}$ or $J = \{ 0 \}$. 
\end{definition}

\begin{remark}
We note that if a ring $R$ admits a topology in which multiplication is continuous (e.g., if $R$ is a $C^*$-algebra), then it is straightforward to show that $R$ is prime if and only if $R$ has the property that whenever $I$ and $J$ are \emph{closed} two-sided ideals of $R$ and $IJ = \{ 0 \}$, then either $I = \{ 0 \}$ or $J = \{ 0 \}$.  Thus a $C^*$-algebra is prime as a ring (in the sense of Definition~\ref{prime-ring-def}) if and only if it is prime as a $C^*$-algebra (in the sense of Definition~\ref{prime-C*-def}).  Moreover, since any $C^*$-algebra has an approximate identity, and any closed two-sided ideal of a $C^*$-algebra is again a $C^*$-algebra, we get that whenever $I$ and $J$ are closed two-sided ideals in a $C^*$-algebra, then $IJ = I \cap J$.  Thus a $C^*$-algebra $A$ is prime if and only if whenever $I$ and $J$ are closed two-sided ideals in $A$ and $I \cap J = \{ 0 \}$, then either $I = \{ 0 \}$ or $J = \{ 0 \}$.   In addition, the existence of an approximate identity in a $C^*$-algebra implies that ideals of ideals are ideals.  In other words, if $A$ is a $C^*$-algebra, $I$ is a closed two-sided ideal of $A$, and $J$ is a closed two-sided ideal of $I$, then $J$ is a closed two-sided ideal of $A$.  Consequently, if $I$ is a closed two-sided ideal in a prime $C^*$-algebra, then $I$ is a prime $C^*$-algebra.
\end{remark}

\begin{definition}
Recall that for a ring $R$ a left $R$-module ${}_RM$ consists of an abelian group $M$ and a ring homomorphism $\pi : R \to \operatorname{End} (M)$ for which $\pi(R)(M) = M$, giving the module action $r \cdot m := \pi(r)(m)$.  We say that ${}_RM$ is faithful if the homomorphism $\pi$ is injective, and we say ${}_RM$ is \emph{simple} if $M \neq \{0\}$ and $M$ has no nonzero proper $R$-submodules; i.e., there are no nonzero proper subgroups $N \subseteq M$ with $\pi(r)(n) \subseteq N$ for all $r \in R$ and $n \in N$.  We make similar definitions for right $R$-modules $M_{R}$, faithful right $R$-modules, and simple right $R$-modules.
\end{definition}

\begin{definition} \label{primitive-ring-def}
Let $R$ be a ring.  We say that $R$ is \emph{left primitive} if there exists a faithful simple left $R$-module.  We say that $R$ is \emph{right primitive} if there exists a faithful simple right $R$-module.
\end{definition}

There are rings that are primitive on one side but not on the other. The first example was constructed by Bergman \cite{Berg, Berg-errata}.  In his ring theory textbook \cite[p.159]{Row}  Rowen also describes another example found by Jategaonkar that displays this distinction.

\begin{definition}   
If $A$ is a $C^*$-algebra, a \emph{$*$-representation} is a $*$-homomorphism $\pi : A \to B(\Hi)$ from the $C^*$-algebra $A$ into the $C^*$-algebra $B(\Hi)$ of bounded linear operators on some Hilbert space $\Hi$.  We say that a $*$-representation $\pi : A \to B(\Hi)$ is \emph{faithful} if it is injective.  For any subset $S \subseteq \Hi$, we define  $\pi(A)S := \clspan \{ \pi(a) h : a \in A \text{ and } h \in S \}$.  If $S = \{ h \}$ is a singleton set, we often write $\pi(A)h$ in place of $\pi(A)\{ h \}$.  A subspace $\mathcal{K} \subseteq \Hi$ is called a $\pi$-\emph{invariant subspace} (or just an \emph{invariant subspace})  if $\pi(a) k \in \mathcal{K}$ for all $a \in A$ and for all $k \in \mathcal{K}$.  Observe that a closed subspace $\mathcal{K} \subseteq \Hi$ is invariant if and only if $\pi(A) \mathcal{K} \subseteq \mathcal{K}$.
\end{definition}

\begin{definition}

Let $A$ be a $C^*$-algebra, and let $\pi : A \to B(\Hi)$ be a $*$-representation.

\begin{itemize}
\item[(1)] We say $\pi$ is a \emph{countably generated $*$-representation} if there exists a countable subset $S \subseteq \Hi$ such that $\pi (A) S = \Hi$.
\item[(2)] We say $\pi$ is a \emph{cyclic $*$-representation} if there exists $h \in H$ such that $\pi(A) h = H$.  (Note that a cyclic $*$-representation could also be called a \emph{singly generated $*$-representation}.)
\item[(3)] We say $\pi$ is an \emph{irreducible $*$-representation} if there are no closed invariant subspaces of $\Hi$ other than $\{ 0 \}$ and $\Hi$.  We note that $\pi$ is irreducible if and only if $\pi(A) h = \Hi$ for all $h \in \Hi \setminus \{ 0 \}$. 
\end{itemize}
\end{definition}

\begin{remark}
One can easily see that for any $*$-representation of a $C^*$-algebra, the following implications hold: $$\text{irreducible} \implies \text{cyclic} \implies \text{countably generated}.$$  In addition, if $\Hi$ is a separable Hilbert space (and hence $\Hi$ has a countable basis), then $\pi$ is automatically countably generated. 
\end{remark}

\begin{definition}  \label{primitive-C*-def}
A $C^*$-algebra $A$ is \emph{primitive} if there exists a faithful irreducible $*$-representation $\pi : A \to B(\mathcal{H})$ for some Hilbert space $\mathcal{H}$. 
\end{definition}

The following result is well known among $C^*$-algebraists.

\begin{proposition} \label{primitive-ring-C*-same}
If $A$ is a $C^*$-algebra, then the following are equivalent.
\begin{itemize}
\item[(1)] $A$ is a left primitive ring (in the sense of Definition~\ref{primitive-ring-def}).
\item[(2)] $A$ is a right primitive ring (in the sense of Definition~\ref{primitive-ring-def}).
\item[(3)] $A$ is a primitive $C^*$-algebra (in the sense of Definition~\ref{primitive-C*-def}).
\end{itemize}
\end{proposition}

\begin{proof}
If $A$ is a $C^*$-algebra and $A^\textrm{op}$ denotes the opposite ring of $A$, then we see that $a \mapsto a^*$ is a ring isomorphism from $A$ onto $A^\textrm{op}$.  Hence $A-\operatorname{Mod}$ and $A^{op}-\operatorname{Mod} \cong \operatorname{Mod}-A$ are equivalent categories.   (We mention that, in general, a $C^*$-algebra is not necessarily isomorphic as a $C^*$-algebra to its opposite $C^*$-algebra; see \cite{PV}.)

The equivalence of (1) and (3) is proven in \cite[Theorem~2.9.7, p.57]{DixmierBook} and uses the results of \cite[Corollary~2.9.6(i), p.57]{DixmierBook} and \cite[Corollary~2.8.4, p.53]{DixmierBook}, which show two things: (i) that any algebraically irreducible representation of $A$ on a complex vector space is algebraically equivalent to a topologically irreducible $*$-representation of $A$ on a Hilbert space, and (ii) that any two topologically irreducible $*$-representations of $A$ on a Hilbert space are algebraically isomorphic if and only if they are unitarily equivalent.
\end{proof}

\begin{remark}
In light of Proposition~\ref{primitive-ring-C*-same}, a $C^*$-algebra is \emph{primitive} in the sense of Definition~\ref{primitive-C*-def} if and only if it satisfies any (and hence all) of the three equivalent conditions stated in Proposition~\ref{primitive-ring-C*-same}.
\end{remark}

\section{Prime and primitive graph $C^*$-algebras}

In this section we establish graph conditions that characterize when the associated graph $C^*$-algebra is prime and when the associated graph $C^*$-algebra is primitive.

The following proposition was established in \cite[Theorem 10.3]{K} in the context of topological graphs, but for the convenience of the reader we provide a streamlined proof here for the context of graph $C^*$-algebras.  We also mention that this same result was obtained for countable graphs in \cite[Proposition~4.2]{BHRS}.)  

\begin{proposition}\label{primenessproposition}   
Let $E$ be any graph.  Then $C^*(E)$ is prime if and only if the following two properties hold:
\begin{itemize}
\item[(i)] $E$ satisfies Condition~(L), and
\item[(ii)] $E$ is downward directed.   
\end{itemize}
\end{proposition}

\begin{proof}
First, let us suppose that $E$ satisfies properties (i) and (ii) from above.  If $I$ and $J$ are nonzero ideals in $C^*(E)$, then it follows from property~(i) and the Cuntz-Krieger Uniqueness Theorem for graph $C^*$-algebras that there exists $u,v \in E^0$ such that $p_u \in I$ and $p_v \in J$.  By property~(ii) there is a vertex $w \in E^0$ such that $u \geq w$ and $v \geq w$.  It follows from Lemma~\ref{reachability-implies-in-I-lem} that $p_w \in I$ and $p_w \in J$.  Hence $0 \neq p_w = p_w  p_w \in IJ$,   and so $C^*(E)$ is a prime $C^*$-algebra.

For the converse, let us suppose that $C^*(E)$ is prime and establish properties (i) and (ii).  Suppose $C^*(E)$ is prime, and $E$ does not satisfy Condition~(L).  Then there exists a cycle $\alpha = e_1 \ldots e_n$ in $E$ that has no exits.  Since $\alpha$ has no exits, the set $H := \alpha^0 = \{ s(e_1), r(e_1), \ldots, r(e_{n-1}) \}$ is a hereditary subset of $E$, and it follows from \cite[Proposition~3.4]{BHRS} that the ideal $I_H = I_{\overline{H}}$ is Morita equivalent to the $C^*$-algebra of the graph $E_H := (H, s^{-1}(H), r|_H, s|_H)$.  Since $E_H$ is the graph consisting of a single cycle, $C^*(E_H) \cong M_n(C(\T))$ for some $n \in \N$ (see \cite[Lemma~2.4]{aHR}).  Therefore, the ideal $I_H$ is Morita equivalent to $C(\T)$.  However, since $C^*(E)$ is a prime $C^*$-algebra, it follows that the ideal $I_H$ is a prime $C^*$-algebra.  Because Morita equvalence preserves primality, and $I_H$ is Morita equivalent to $C(\T)$, it follows that $C(\T)$ is a prime $C^*$-algebra.  However, it is well known that $C(\T)$ is not prime: If $C$ and $D$ are proper closed subsets of $\T$ for which $C \cup D = \T$, and we set $I_C := \{ f \in C(\T) : f|_C \equiv 0 \}$ and $I_D := \{ f \in C(\T) : f|_D \equiv 0 \}$, then $I_C$ and $I_D$ are ideals with $I_C \cap I_D = \{ 0 \}$ but $I_C \neq \{ 0 \}$ and $I_D \neq \{ 0 \}$.  This provides a contradiction, and we may conclude that $E$ satisfies Condition~(L) and that property (i) holds.

Next, suppose $C^*(E)$ is prime.  Let $u, v \in E^0$, and consider $H(u) := \{ x \in E^0 : u \geq x \}$ and $H(v) := \{ x \in E^0 : v \geq x \}$.  Then $H(u)$ and $H(v)$ are nonempty hereditary subsets, and the ideals $I_{H(u)} = I_{\overline{H(u)}} $ and $I_{H(v)}=I_{\overline{H(v)}}$ are each nonzero.  Since $C^*(E)$ is a prime $C^*$-algebra, it follows that $I_{ \overline{H(u)} \cap \overline{H(v)}} = I_{\overline{H(u)}} \cap I_{\overline{H(v)}} \neq \{ 0 \}$.  Thus $\overline{H(u)} \cap \overline{H(v)} \neq \emptyset$, and Lemma~\ref{sat-int-lem} implies that $H(u) \cap H(v) \neq \emptyset$.  If we choose $w \in H(u) \cap H(v)$, then $u \geq w$ and $v \geq w$.  Hence $E$ is downward directed and property (ii) holds.
\end{proof}

The following is well known, but we provide a proof for completeness.

\begin{lemma}\label{primitiveimpliesprime}   Any  primitive $C^*$-algebra  is prime. 
\end{lemma}

\begin{proof}
Let $A$ be a primitive $C^*$-algebra.  Then there exists a faithful irreducible $*$-representation $\pi : A \to B(\Hi)$.  Let $I$ and $J$ be ideals of $A$, and suppose that $I J = \{ 0 \}$.  If $J \neq \{ 0 \}$, then the faithfulness of $\pi$ implies $\pi(J) \Hi \neq \{0\}$, and the fact that $J$ is an ideal shows that $\pi(J) \Hi$ is a closed invariant subspace for $\pi$.  Since $\pi$ is irreducible, it follows that $\pi(J) \Hi = \Hi$.  Using the fact that $IJ = \{ 0 \}$, it follows that
$\{0\} = \pi(IJ)\mathcal{H} = \pi(I)\pi(J)\mathcal{H} =\pi(I)\mathcal{H}$.  Since $\pi$ is faithful, this implies that $I = \{ 0 \}$.  Thus $A$ is a prime $C^*$-algebra. 
\end{proof}

The following two lemmas are elementary, but very useful.

\begin{lemma} \label{uncountable-sum-inf-lem}
If $X \subseteq (0, \infty)$ is an uncountable subset of positive real numbers, then the sum $\displaystyle \sum_{x \in X} x$ diverges to infinity.
\end{lemma}

\begin{proof}
For each $n \in \N$, let $X_n := \{ x \in X : x \geq \frac{1}{n} \}$.  Then $X = \bigcup_{n=1}^\infty X_n$, and since $X$ is uncountable, there exists $n_0 \in \N$ such that $X_{n_0}$ is infinite.  Thus $\sum_{x \in X} x \geq \sum_{x \in X_{n_0}} x \geq \sum_{i=1}^\infty \frac{1}{n} = \infty$.
\end{proof}

\begin{lemma} \label{nouncountablesetsoforthogidempotents}
Let $A$ be a $C^*$-algebra, and let $\{p_i \  |  \ i\in I\}$ be a set of nonzero mutually orthogonal projections in $A$.  If there exists a $*$-representation $\pi : A \to B(\Hi)$ and a vector $\xi \in \Hi$ such that $\pi(p_i) \xi \neq 0$ for all $i \in I$, then $I$ is at most countable.
\end{lemma}

\begin{proof}
Let $P := \bigoplus_{i \in I} \pi(p_i)$ be the projection onto the direct sum of the images of the $\pi(p_i)$.  Since the $\pi (p_i)$ are mutually orthogonal projections, the Pythagorean theorem shows $\sum_{i \in I} \| \pi(p_i) \xi \|^2 = \| P \xi \|^2 \leq \| \xi \|^2 < \infty.$  Since each $\| \pi(p_i) \xi\|^2$ term is nonzero, and since any uncountable sum of positive real numbers diverges to infinity (see Lemma~\ref{uncountable-sum-inf-lem}), it follows that the index set $I$ is at most countable.
\end{proof}

The following proposition provides a necessary condition for a graph $C^*$-algebra to be primitive.

\begin{proposition}
If $E$ is a graph and $C^*(E)$ has a faithful countably generated $*$-representation, then $E$ satisfies the Countable Separation Property.
\end{proposition}

\begin{proof}
By hypothesis there is a faithful countably generated $*$-representation $\pi : C^*(E) \to B(\Hi)$.  Thus there exists a countable set of vectors $S := \{ \xi_i \}_{i=1}^\infty \subseteq \Hi$ with $\pi(C^*(E))S = \Hi$.  For every $n \in \N \cup \{ 0 \}$ and for every $i \in \N$, define
$$\Gamma_{n, i} := \{ \alpha \in \path (E) : |\alpha|=n \text{ and } \pi(s_\alpha s_\alpha^*) \xi_i \neq 0 \}.$$
(Recall that we view vertices as paths of length zero, and in this case $s_v = p_v$.)  By Remark \ref{productsofsalphasalphastar}, for any $n \in \N \cup \{ 0 \}$ the set $\{ s_\alpha s_\alpha^* : \alpha \in \path (E) \text{ and } |\alpha| = n \}$ consists of mutually orthogonal projections, and hence Lemma~\ref{nouncountablesetsoforthogidempotents} implies that for any $n \in \N \cup \{ 0 \}$ and for any $i \in \N$, the set $\Gamma_{n,i}$ is countable.

Define $$\Gamma := \bigcup_{n=0}^\infty \bigcup_{i=1}^\infty \Gamma_{n,i},$$
which is countable since it is the countable union of countable sets.  Also define
$$\Theta := \bigcup_{\alpha \in \Gamma} U(r(\alpha)).$$
Then $\Theta \subseteq E^0$ is a set of vertices, and we shall show that $\Theta = E^0$.   We note that $\Theta$ consists precisely of the vertices $v$ in $E$ for which there is a path from $v$ to $w$, where $w=r(\alpha)$ for a path $\alpha$ having the property that $\pi(S_\alpha S_\alpha^*)\xi_i$ is nonzero for some $i$.  

Let $v \in E^0$, and let $I$ denote the closed two-sided ideal of $C^*(E)$ generated by $p_v$.  
Since $I$ is a nonzero ideal and $\pi$ is faithful, it follows that $\pi(I) \Hi \neq \{ 0 \}$.  Thus
$$\pi(I) S = \pi(I C^*(E)) S = \pi(I) \pi(C^*(E)) S = \pi (I) \Hi \neq \{ 0 \}$$
and hence there exists $a \in I$ and $\xi_i \in S$ such that $\pi(a) \xi_i \neq 0$.  If $H(v) := \{ w \in E^0 : v \geq w \}$ is the hereditary subset of $E^0$ generated by $v$ (as given in Definition \ref{sat-hered-def}), then it follows from \cite[\S3]{BHRS} that 
$$I = I_{H(v)} = \clspan \{ s_\alpha s_\beta^* : \alpha, \beta \in \path (E) \text{ and } r(\alpha) = r(\beta) \in H(v) \}.$$  Since $a \in I$ and $\pi(a) \xi_i \neq 0$, it follows from the linearity and continuity of $\pi$ that there exists $\alpha, \beta \in \path(E)$ with $r(\alpha) = r(\beta) \in H(v)$ and $\pi(s_\alpha s_\beta^*) \xi_i \neq 0$.  Hence
$$\pi (s_\alpha s_\beta^*) \pi(s_\beta s_\beta^*) \xi_i = \pi(s_\alpha s_\beta^* s_\beta s_\beta^*) \xi_i = \pi(s_\alpha s_\beta^*) \xi_i \neq 0,$$ and thus $\pi(s_\beta s_\beta^*) \xi_i \neq 0$.  If we let $n := |\beta|$, then $\beta \in \Gamma_{n,i} \subseteq \Gamma$.  Since $r(\beta) \in H(v)$, it follows that $v \geq r(\beta)$ and $v \in U(r(\beta)) \subseteq  \bigcup_{\alpha \in \Gamma} U(r(\alpha)) = \Theta$.  Thus $E^0 = \Theta := \bigcup_{\alpha \in \Gamma} U(r(\alpha))$, and since $\Gamma$ is countable, $E$ satisfies the Countable Separation Property.
\end{proof}

\begin{corollary}
If $E$ is a graph and $C^*(E)$ has a faithful cyclic $*$-representation, then $E$ satisfies the Countable Separation Property.
\end{corollary}

\begin{corollary} \label{primitive-implies-CSP-cor}
If $E$ is a graph and $C^*(E)$ is primitive, then $E$ satisfies the Countable Separation Property.
\end{corollary}

Our main objective in this article is the following result, in which we  provide necessary and sufficient conditions on a graph $E$ for the $C^*$-algebra $C^*(E)$ to be primitive.  In particular, we identify the precise additional condition on $E$ that guarantees a prime graph $C^*$-algebra $C^*(E)$ is primitive.  With the previously mentioned result of Dixmier  \cite[Corollaire 1]{DixmierPaper} as context (i.e.,  that any prime separable $C^*$-algebra is primitive), we note that our additional condition is not a separability hypothesis on $C^*(E)$, but rather the Countable Separation Property of $E$.

\begin{theorem} \label{primitivitytheorem}
Let $E$ be any graph.  Then $C^*(E)$ is primitive if and only if the following three properties hold:
\begin{itemize}
\item[(i)] $E$ satisfies Condition~(L),
\item[(ii)] $E$ is downward directed, and
\item[(iii)] $E$ satisfies the Countable Separation Property.
\end{itemize}
In other words, by Proposition~\ref{primenessproposition}, $C^*(E)$ is primitive if and only if $C^*(E)$ is prime and $E$ satisfies the Countable Separation Property. 
\end{theorem}

\begin{proof}  To prove this result we establish both the sufficiency and the necessity of properties (i), (ii), and (iii) for $C^*(E)$ to be primitive.

\smallskip

\noindent  \textsc{Proof of Sufficiency:}   Since $E$ satisfies the Countable Separation Property by (iii), there exists a countable set $X \subseteq E^0$ such that $E^0 = \bigcup_{x \in X} U(x)$.  Since $X$ is countable, we may list the elements of $X$ as $X := \{v_1, v_2, \ldots \}$, where this list is either finite or countably infinite.  For convenience of notation, let us write $X = \{ v_i \}_{i \in I}$ where the indexing set $I $ either has the form $I = \{ 1, \ldots, n \}$ or $I = \N$.

We inductively define a sequence of paths $\{ \lambda_i \}_{i \in I} \subseteq \path (E)$ satisfying the following two properties:
\begin{itemize}
\item[(a)] For each $i \in I$ we have $v_i \geq r(\lambda_i)$.
\item[(b)] For each $i \in I$ with $i \geq 2$ there exists $\mu_{i} \in \path (E)$ such that $\lambda_{i} = \lambda_{i-1} \mu_{i}$.
\end{itemize}

To do so, define $\lambda_1 := v_1$ and note that for $i=1$ Property~(a) is satisfied trivially and Property~(b) is satisfied vacuously.  Next suppose that $\lambda_1, \ldots \lambda_n$ have been defined so that Property~(a) and Property~(b) are satisfied for $1 \leq i \leq n$.   By hypothesis (ii) $E$ is downward directed, and hence there exists a vertex $u_{n+1}$ in $E$ such that $r(\lambda_n) \geq u_{n+1}$ and $v_{n+1}\geq u_{n+1}$.
 Let $\mu_{n+1}$ be a path from $r(\lambda_n)$ to $u_{n+1}$, and define $\lambda_{n+1} :=\lambda_{n}\mu_{n+1}$.  Then the paths $\lambda_1, \ldots, \lambda_{n+1}$ satisfy Property~(a) and Property~(b) for all $1 \leq i \leq n$.  Continuing in this manner, we either exhaust the elements of $I$ or inductively create a sequence of paths $\{ \lambda_i \}_{i \in I} \subseteq \path (E)$ satisfying Property~(a) and Property~(b) for all $i \in I$. 
 
  Note that, in particular, for each $i < n$, we have $\lambda_n = \lambda_i \mu_{i+1} \ldots \mu_n$, and hence by Remark \ref{productsofsalphasalphastar} for $i\leq n$ we have 
 $$s_{\lambda_i} s_{\lambda_i}^* s_{\lambda_n} s_{\lambda_n}^* = s_{\lambda_n} s_{\lambda_n}^*.$$

 We now establish for future use that every nonzero closed two-sided ideal $J$ of $C^*(E)$ contains  $s_{\lambda_n}s_{\lambda_n}^{\ast} $ for some $n \in I$.   Using hypothesis (i), our graph satisfies Condition~(L), and the Cuntz-Krieger Uniqueness Theorem (Theorem~\ref{CKUT-thm}) implies that there exists $w \in E^0$ such that $p_w \in J$.  By the Countable Separation Property  there exists $v_n \in X$ such that $w\geq v_n$.   
In addition, by Property~(a) above $v_n \geq r(\lambda_n)$, so that there is a path $\gamma$ in $E$ for which $s(\gamma)=w$ and $r(\gamma)=r(\lambda_n)$.  Since $p_w\in J$, we have by Lemma  \ref{reachability-implies-in-I-lem} that  $p_{r(\lambda_n)}  \in J$, so that $s_{\lambda_n}s_{\lambda_n}^* = s_{\lambda_n} p_{r(\lambda_n)} s_{\lambda_n}^*  \in J$ as desired.  

Define 
 $$L :=  \left\{ \sum_{i=1}^n (x_i - x_is_{\lambda_i}s_{\lambda_i}^{\ast}) :  n \in I \text{ and } x_1, \ldots, x_n \in C^*(E) \right\}.$$
Recall that $\lambda_1 = v_1$, and by our convention (see Definition~\ref{s-alpha-def}) $s_{\lambda_1} = p_{v_1}$ and $s_{\lambda_1} s_{\lambda_1}^* = p_{v_1}$.  Clearly $L$ is an algebraic (i.e., not necessarily closed) left ideal of $C^*(E)$.  We claim that $p_{v_1} \notin L$.  For otherwise  there would exist $n\in I$ and $x_1, ..., x_n \in C^*(E)$ with 
$$\sum_{i=1}^n (x_i - x_i s_{\lambda_i}s_{\lambda_i}^*) = p_{v_1} .$$  But then multiplying both sides of this equation by $s_{\lambda_n}s_{\lambda_n}^{\ast}$ on the right gives $$\sum_{i=1}^n (x_i s_{\lambda_n}s_{\lambda_n}^{\ast} - x_i s_{\lambda_i}s_{\lambda_i}^{\ast}s_{\lambda_n}s_{\lambda_n}^{\ast}) = s_{\lambda_n}s_{\lambda_n}^*.$$
Using the previously displayed observation,   the above equation becomes
$$\sum_{i=1}^n (x_i s_{\lambda_n}s_{\lambda_n}^{\ast} - x_i s_{\lambda_n}s_{\lambda_n}^*) =   s_{\lambda_n}s_{\lambda_n}^*,$$ which gives $0 =  s_{\lambda_n}s_{\lambda_n}^{\ast}$, a contradiction.    Hence we may conclude that $p_{v_1} \notin L$, and $L$ is a proper left ideal of $C^*(E)$.

By the definition of $L$, we have  $a - ap_{v_1} \in L$  for each $a\in C^*(E)$ and hence $L$ is a modular left ideal of $C^*(E)$, a property which necessarily passes to any left ideal of $C^*(E)$ containing $L$.   (See \cite{Con} for definitions of appropriate terms.  Also, cf.~\cite[Chapter 2, Theorem~2.1.1]{Ric}.)  Let $\mathcal{T}$ denote the set of (necessarily modular) left ideals of $C^*(E)$ that contain $L$ but do not contain $p_{v_1}$.   Since $L \in \mathcal{T}$, we have $\mathcal{T} \neq \emptyset$.   By a Zorn's Lemma argument there exists a maximal element in $\mathcal{T}$, which we denote $M$.   We claim that $M$ is a maximal left ideal of $C^*(E)$.  For suppose that $M'$ is a left ideal of $C^*(E)$ having $M \subsetneqq M'$.  Then by the maximality of $M$ in $\mathcal{T}$ we have $p_{v_1} \in M'$, so that $x p_{v_1} \in M'$ for each $x\in C^*(E)$, which gives that $x = (x - xp_{v_1}) + xp_{v_1}  \in L + M' = M'$ for each $x\in C^*(E)$, so that $M' = C^*(E)$.   Thus $M$ is a maximal left ideal, and since $M$ is also modular, $M$ is a maximal modular left ideal.  Thus by \cite[VII.2, Exercise~6]{Con} (cf.~\cite[Chapter~2, Corollary~2.1.4]{Ric}) $M$ is closed.

Since $M$ is closed we may form the regular $*$-representation of $C^*(E)$ into $C^*(E)/M$; i.e., the homomorphism $\pi : C^*(E) \to \operatorname{End} (C^*(E)/M)$ given by $\pi(a) (b+M) := ab + M$.  In this way $C^*(E)/M$ becomes a left $C^*(E)$-module.  The submodules of $C^*(E)/M$ correspond to left ideals of $C^*(E)$ containing $M$, and by the maximality of $M$ the only submodules of $C^*(E)/M$ are $\{ 0 \}$ and $M$.  Hence $C^*(E)/M$ is a simple module.  We claim that $\ker \pi = \{ 0 \}$.  If $a \in \ker \pi$, then for all $b \in C^*(E)$ we have $ab +M = \pi(a) (b+M) = 0+M$, so that $ab \in M$ for all $b \in C^*(E)$.  If $\{ e_\lambda \}_{\lambda \in \Lambda}$ is an approximate unit for $C^*(E)$, then the previous sentences combined with the fact that $M$ is closed implies that $a = \lim_{\lambda} a e_\lambda \in M$.  Hence we have established that $\ker \pi \subseteq M$.  Furthermore, since $\ker \pi$ is a closed two-sided ideal of $C^*(E)$, it follows from above that if $\ker \pi$ is nonzero, then $s_{\lambda_n} s_{\lambda_n}^* \in \ker \pi$ for some $n \in I$.  But then $M$ contains $a s_{\lambda_n}s_{\lambda_n}^*$ for every $a\in C^*(E)$.  In addition, since $a - as_{\lambda_n}s_{\lambda_n}^* \in L \subseteq M$ for all $a \in C^*(E)$, it follows that $a \in M$ for all $a \in C^*(E)$, and $C^*(E) \subseteq M$, which is a contradiction.  We conclude that $\ker \pi = \{0\}$.  Hence $C^*(E)/M$ is a faithful left $C^*(E)$-module, which with the simplicity of $C^*(E)/M$ yields that  $C^*(E)$ is primitive as a ring.  It follows (see Remark~\ref{primitive-ring-C*-same}) that $C^*(E)$ is primitive as a $C^*$-algebra.

\smallskip 

\noindent  \textsc{Proof of Necessity:}  If  $C^*(E)$ is primitive, then $C^*(E)$ is necessarily prime by Lemma~\ref{primitiveimpliesprime}, so by Proposition~\ref{primenessproposition} we get that $E$ satisfies Condition~(L) and  is downward directed.  In addition, Corollary~\ref{primitive-implies-CSP-cor} shows that if $C^*(E)$ is primitive, then $E$ satisfies the Countable Separation Property, thus completing the proof.
\end{proof}

\begin{corollary} \label{primitive-equiv-cor}
Let $E$ be a graph. Then the following are equivalent:
\begin{itemize}
\item[(i)] $C^*(E)$ is primitive.
\item[(ii)] $C^*(E)$ is prime and $E$ satisfies the Countable Separation Property.
\item[(iii)] $C^*(E)$ is prime and $C^*(E)$ has a faithful cyclic $*$-representation.
\item[(iv)] $C^*(E)$ is prime and $C^*(E)$ has a faithful countably generated $*$-representation.
\end{itemize}
\end{corollary}

\begin{proof}
The equivalence of (i) and (ii) follows from Theorem~\ref{primitivitytheorem} and Proposition~\ref{primenessproposition}.  The implications (i) $\implies$ (iii) $\implies$ (iv) are trivial.  The implication (iv) $\implies$ (ii) follows from Corollary~\ref{primitive-implies-CSP-cor}.
\end{proof}

\begin{remark}
In \cite[Problem~13.6]{K3} Katsura asks whether a prime $C^*$-algebra is primitive if it has a faithful cyclic $*$-representation.   Corollary~\ref{primitive-equiv-cor} shows that the answer to Katsura's question is affirmative in the class of graph $C^*$-algebras.   Moreover, Corollary~\ref{primitive-equiv-cor} prompts us to ask the following more general question:

$ $

\noindent \textbf{Question:} If a $C^*$-algebra is prime and has a faithful countably generated representation, then is that $C^*$-algebra primitive?

$ $

\noindent Again, Corollary~\ref{primitive-equiv-cor} provides us with an affirmative answer to this question for the class of graph $C^*$-algebras.  In addition, an affirmative answer to this question in general implies an affirmative answer to Katsura's question in \cite[Problem~13.6]{K3}.
\end{remark}

\begin{remark}
In the introduction of \cite{W} Weaver notes that his example of a prime and not primitive $C^*$-algebra ``$\ldots$ places competing demands on the set of partial isometries: it must be sufficiently abundant $\ldots$ and [simultaneously] sufficiently sparse $\ldots$"    Effectively, Theorem~\ref{primitivitytheorem}   identifies  precisely how these two competing demands play out   in the context of a graph $C^*$-algebra $C^*(E)$:  primeness (abundance of partial isometries) corresponds to $E$ satisfying Condition~(L) and being downward directed, while nonprimitivity (sparseness  of partial isometries) corresponds to $E$ not satisfying the Countable Separation Property.  
\end{remark}

\begin{remark}
It is shown in \cite[Proposition 4.2]{BHRS} that for a graph $E$ having both $E^0$ and $E^1$ countable, $C^*(E)$ is primitive if and only if $E$ is downward directed and satisfies Condition~(L).   When $E^0$ is countable, then $E$ trivially satisfies the Countable Separation Property, with $E^0$ itself providing the requisite countable set of vertices.   Thus Theorem~\ref{primitivitytheorem} provides both a generalization of (since the countability of $E^1$ is not required), and an alternate proof for, the result in \cite[Proposition 4.2]{BHRS}.
\end{remark}

\section{Examples of prime but not primitive $C^*$-algebras}

We now offer a number of examples of prime nonprimitive $C^*$-algebras that arise from the characterizations presented in Proposition \ref{primenessproposition} and  Theorem~\ref{primitivitytheorem}.   These examples are similar in flavor to the classes of examples that played an important role in \cite{ABR}.

\begin{definition}\label{EsubTsubXdefinition}
Let $X$ be any nonempty set, and let $\mathcal{F}(X)$ denote the set  of  finite nonempty  subsets of $X$.   We  define three graphs $E_{A}(X)$, $E_{L}(X)$, and $E_{K}(X)$ as follows.

\medskip

\noindent (1) The graph $E_{A}(X)$ is defined by
$$E_{A}(X)^0 \ = \ \mathcal{F}(X) \ \  \text{ and } \ \ 
E_{A}(X)^1 = \{e_{A,A'} : A,A'\in \mathcal{F}(X) \ \mbox{and} \ A\subsetneqq A'\},$$
with $s(e_{A,A'})=A$ and $r(e_{A,A'})=A'$ for each $e_{A,A'} \in E_{A}(X)^1.$

\medskip

\noindent (2) The graph $E_{L}(X)$ is defined by
$$E_{L}(X)^0 \ = \ \mathcal{F}(X) \ \ \text{ and } \ \
E_{L}(X)^1 = \{e_{A,A'} : A,A'\in \mathcal{F}(X) \ \mbox{and} \ A\subseteqq A'\},$$
with $s(e_{A,A'})=A$ and $r(e_{A,A'})=A'$ for each $e_{A,A'} \in E_{L}(X)^1$.

\medskip

\noindent (3) The graph $E_{K}(X)$ is defined by
\begin{align*}
E_{K}(X)^0 &=  \mathcal{F}(X) \quad \text{ and } \\
E_{K}(X)^1 &= \{e_{A,A'} : A,A'\in \mathcal{F}(X) \ \mbox{and} \ A\subseteqq A'\} \cup \{f_A : A\in \mathcal{F}(X)\},
\end{align*}
with  $s(e_{A,A'})=A$, and $r(e_{A,A'})=A'$ for each $e_{A,A'} \in E_{K}(X)^1,$ and with $s(f_A) = r(f_A) = A$ for all $f_A$. 
\end{definition}

\medskip

\begin{remark}\label{remarkaboutthreegraphs}
Observe that for any nonempty set $X$, the graph $E_{A}(X)$ is acyclic, the graph $E_{L}(X)$ has as its only simple cycles the single loop at each vertex (so that $E_{L}(X)$ satisfies Condition (L), but not Condition (K)), and the graph $E_{K}(X)$ has as its only simple cycles the two loops at each  vertex (so that $E_{K}(X)$ satisfies Condition (K)).
\end{remark}

In this section all direct limits that we discuss will be direct limits in the category of $C^*$-algebras, so that the direct limit algebras discussed are $C^*$-algebras.  An AF-algebra is typically defined to be a $C^*$-algebra that is the direct limit of a sequence of finite-dimensional algebras.  Since it is a sequential direct limit, an AF-algebra is necessarily separable.  Some authors, such as Katsura in \cite{K}, have considered arbitrary direct limits of finite-dimensional $C^*$-algebras.  Following Katsura in \cite{K}, we shall define an \emph{AF-algebra} to be a $C^*$-algebra that is the direct limit of finite-dimensional $C^*$-algebras.  (Equivalently, an AF-algebra is a $C^*$-algebra $A$ with a directed family of finite dimensional $C^*$-subalgebras whose union is dense in $A$.)  An AF-algebra in our sense is a sequential direct limit (i.e., an AF-algebra in the traditional sense) if and only if it is separable.

\begin{lemma}\label{downdirandCSPpropertieslemma}
Let $E_A(X), E_L(X),$ and $E_K(X)$ be the graphs presented in Definition \ref{EsubTsubXdefinition}. 
\begin{enumerate}
\item Each of  the graphs $E_A(X), E_L(X),$ and $E_K(X)$ is downward directed.
\item Each of  the graphs $E_A(X), E_L(X),$ and $E_K(X)$  satisfies the Countable Separation Property if and only if $X$ is countable. 
\end{enumerate}
\end{lemma}

\begin{proof}
To establish (1),  we observe that in each of the graphs $E_A(X), E_L(X),$ and $E_K(X)$,   each pair of vertices $A,A'$ corresponds to a pair of finite subsets of $X$, so that if $B$ denotes the finite set  $A \cup A'$ then there is an edge $e_{A,B}$ from $A$ to $B$ and an edge $e_{A',B}$ from $A'$ to $B$.   Downward directedness of each of the three graphs follows.   

For (2), we note that if $X$ is countable then $\mathcal{F}(X)$ is countable, so that in this case each of the three graphs has countably many vertices, and thus trivially satisfies the Countable Separation Property.  On the other hand, if $X$ is uncountable, then any countable union of elements of  $\mathcal{F}(X)$ includes only countably many elements of $X$, so that there exists some vertex (indeed, uncountably many vertices) which does not connect to the vertices represented by such a countable union.   
\end{proof}

\begin{proposition}\label{graphproperties}
Let $X$ be an infinite set, and let $E_A(X)$, $E_L(X)$, and $E_K(X)$ be the graphs presented in Definition~\ref{EsubTsubXdefinition}.  Then
\begin{enumerate}
\item    $C^*(E_{A}(X))$ is a prime AF-algebra for any set $X$.  Furthermore,   $C^*(E_{A}(X))$  is primitive if and only if $X$ is countable.  In addition, $C^*(E_{A}(X))$ is a separable AF-algebra if and only if $X$ is countable.
\item   $C^*(E_{L}(X))$ is a prime $C^*$-algebra that is not AF for any set $X$.  Furthermore,  $C^*(E_{L}(X))$  is primitive if and only if $X$ is countable.  In addition, $C^*(E_{L}(X))$ contains an ideal that is not gauge invariant.
\item  $C^*(E_K(X))$ is a prime $C^*$-algebra of real rank zero that is not AF for any set $X$.   Furthermore, $C^*(E_K(X))$  is primitive if and only if $X$ is countable.  In addition, every ideal of $C^*(E_{K}(X))$ is gauge invariant.
\end{enumerate}
\end{proposition}

\begin{proof}

 The indicated primeness and primitivity properties of each of the three algebras $C^*(E_A(X))$, $C^*(E_L(X))$, and $C^*(E_K(X))$ follow directly from Proposition \ref{primenessproposition}, Theorem \ref{primitivitytheorem}, and Lemma \ref{downdirandCSPpropertieslemma}.   We now take up the discussion of the additional properties. 

(1)  Since $E_A(X)$ is a countable graph if and only if $X$ is countable, it follows that $C^*(E_A(X))$ is separable if and only if $X$ is countable.

(2)  We see that $E_L(X)$ has exactly one loop at each vertex, and that these are the only simple cycles in $E_L(X)$.  Thus every vertex of $E_L(X)$ is the base of exactly one simple cycle, so that  $E_L(X)$ does not satisfy Condition~(K).  In addition, if $\alpha$ is the loop based at the vertex $A$, then since $A$ is finite and $X$ is infinite there exists an element $x \in X \setminus A$, and the edge from $A$ to $A \cup \{ x \}$ provides an exit for $\alpha$.  Hence every cycle in $E_L(X)$ has an exit and $E_L(X)$ satisfies Condition~(L).  Since $E_L(X)$ contains a cycle, $C^*(E_L(X))$ is not AF.   Moreover, since $E_L(X)$ does not satisfy Condition~(K), it follows that $C^*(E_L(X))$ contains an ideal that is not gauge invariant.  (This is established  for row-finite countable graphs in \cite[Theorem~2.1.19]{Tom9}, although the same argument works for non-row-finite or uncountable graphs.  Alternatively, the result for uncountable graphs may also be obtained as a special case of \cite[Theorem~7.6]{K3}.)

(3)     As $E_K(X)$ has two loops at each vertex,  every vertex in $E_K(X)$ is the base point of two distinct simple cycles, so that $E_K(X)$ satisfies Condition~(K).   Since $E_K(X)$ contains a cycle, $C^*(E_K(X))$ is not AF.   In addition, since $E_K(X)$ satisfies Condition~(K), $C^*(E_K(X))$ has real rank zero.  (This was established  for $C^*$-algebras of locally-finite countable graphs in \cite[Theorem~4.1]{JP01} and for $C^*$-algebras of countable graphs in \cite[Theorem~2.5]{HS}, and can be extended to uncountable graphs using the approximation methods of \cite[Lemma~1.2]{RS}.)   Moreover, since $E_K(X)$ satisfies Condition~(K), all ideals of $C^*(E_K(X))$ are gauge invariant. 
\end{proof}

We can now produce infinite classes of $C^*$-algebras that are prime and not primitive.  In fact, we are able to produce three such classes of $C^*$-algebras: one in which all the $C^*$-algebras are AF-algebras, one in which all the $C^*$-algebras are non-AF and contain ideals that are not gauge invariant, and one in which all the $C^*$-algebras are non-AF, have all of their ideals gauge invariant, and are real rank zero.

\begin{proposition}\label{examplesofprimenotprimitive}    
For a set $X$ we let $|X|$ denote the cardinality of $X$.
\begin{enumerate}
\item[(1)]  If $\mathcal{C} := \{ X_i : i \in I \}$ is a collection of sets with $|X_i| > \aleph_0$ for all $i \in I$, and with $|X_i| \neq |X_j|$ for all $i, j \in I$ with $i \neq j$, then
$$\{ C^*(E_A(X_i)) : i \in I \}$$ is a collection of AF-algebras that are prime and not primitive.  Moreover $C^*(E_A(X_i)) $ is not Morita equivalent to $C^*(E_A(X_j))$ for all $i,j \in I$ with $i \neq j$.

\item[(2)]  If $\mathcal{C} := \{ X_i : i \in I \}$ is a collection of sets with $|X_i| \geq 2^{\aleph_0}$ for all $i \in I$, and with $|X_i| \neq |X_j|$ for all $i, j \in I$ with $i \neq j$, then
$$\{ C^*(E_L(X_i)) : i \in I \}$$ is a collection of non-AF $C^*$-algebras each of which is prime and not primitive, and each of which has the property that it contains ideals that are not gauge invariant.  Moreover, $C^*(E_L(X_i))$ is not Morita equivalent to $C^*(E_L(X_j))$ for all $i,j \in I$ with $i \neq j$.

\item[(3)]  If $\mathcal{C} := \{ X_i : i \in I \}$ is a collection of sets with $|X_i| > \aleph_0$ for all $i \in I$, and with $|X_i| \neq |X_j|$ for all $i, j \in I$ with $i \neq j$, then
$$\{ C^*(E_K(X_i)) : i \in I \}$$ is a collection of non-AF $C^*$-algebras of real rank zero each of which is prime and not primitive, and each of which has the property that all of its ideals are gauge invariant.  Moreover, $C^*(E_K(X_i))$ is not Morita equivalent to $C^*(E_K(X_j))$ for all $i,j \in I$ with $i \neq j$.
\end{enumerate}
\end{proposition}

\begin{proof}
For (1), the fact that the $C^*(E_A(X_i))$ are AF-algebras that are prime and not primitive follows from Proposition~\ref{graphproperties}(1).  It remains to show the $C^*(E_A(X_i))$ are mutually non-Morita equivalent.  For each $i \in I$, the graph $E_A(X_i)$ satisfies Condition~(K), and hence the ideals in $C^*(E_A(X_i))$ are in one-to-one correspondence with admissible pairs $(H,S)$, where $H$ is a saturated hereditary subset of $E_A(X_i)^0$ and $S$ is a subset of breaking vertices for $H$.  For each $x \in X_i$ define $H_x := E_A(X_i)^0 \setminus \{ x \}$.  Because $\{ x \}$ is a source and an infinite emitter in $E_A(X_i)$ that only emits edges into $H_x$, the set $H_x$ is saturated and hereditary, and $H_x$ has no breaking vertices.  In addition, any saturated hereditary subset of $E_A(X_i)^0$ that contains $H_x$ is either equal to $H_x$ or equal to $E_A(X_i)^0$.  Thus $I_{(H_x,\emptyset)}$ is a maximal ideal in $C^*(E_A(X_i))$.  Conversely, any maximal ideal must have the form $I_{(H,S)}$ for an admissible pair $(H,S)$, and in order to be a proper ideal there exists $x \in E_A(X_i)^0 \setminus H$.  Because $H \subseteq H_x$, the maximality of $I_{(H,S)}$ implies $H= H_x$ and $S = \emptyset$.  Thus we may conclude that the map $x \mapsto I_{(H_x, \emptyset)}$ is a bijection from $X$ onto the set of maximal ideals of $C^*(E_A(X_i))$.  Hence the cardinality of the set of maximal ideals of $C^*(E_A(X_i))$ is equal to $| X_i|$.  Since any two Morita equivalent $C^*$-algebras have isomorphic lattices of ideals, any two Morita equivalent $C^*$-algebras have sets of maximal ideals with the same cardinality.  Thus, when $i,j \in I$ with $i \neq j$, the fact that $|X_i| \neq |X_j|$ implies that $C^*(E_A(X_i)) $ is not Morita equivalent to $C^*(E_A(X_j))$.  

For (2), the fact that the $C^*$-algebras $C^*(E_L(X_i))$ are non-AF $C^*$-algebras that are prime and not primitive and that each contains ideals that are not gauge invariant follows from Proposition~\ref{graphproperties}(2).  It remains to show the $C^*(E_L(X_i))$ are mutually non-Morita equivalent.  Fix $i \in I$, and let $J \triangleleft C^*(E)$ be a maximal ideal in $C^*(E_L(X_i))$.  Then there exists exactly one $x \in X_i$ such that $I_{H_x} \subseteq J$, where $H_x := E_L(X_i)^0 \setminus \{ x \}$.  (If there did not exist such an $x$, then $J$ would be all of $C^*(E_L(X_i))$, and if there existed more than one $x$, then $J$ would not be maximal.)  Since $I_{H_x} \subseteq J$, it follows that $J$ corresponds to a maximal ideal of the quotient $C^*(E_L(X_j)) / I_{H_x}$.  Because the graph $E_L(X_i) \setminus H_x$ is a single vertex with a single loop, we see that $C^*(E_L(X_j)) / I_{H_x} \cong C(\T)$.  For this maximal ideal there is a unique $z \in \mathcal{T}$ such that the ideal is equal to $\{ f \in C(\T) : f(z) = 0 \}$.  This line of reasoning shows that the map $J \mapsto (x,z)$ is a bijection from the set of maximal ideals of $C^*(E_L(X_i))$ onto the set $X \times \T$.  Since $| X_i | \geq 2^{\aleph_0}$ and $| \T | = 2^{\aleph_0}$, we may conclude that $|X_i \times \T| = |X_i|$.  Thus the set of maximal ideals of $C^*(E_L(X_i))$ has cardinality equal to $| X_i|$.  Since any two Morita equivalent $C^*$-algebras have isomorphic lattices of ideals, any two Morita equivalent $C^*$-algebras have sets of maximal ideals with the same cardinality.  Thus, when $i,j \in I$ with $i \neq j$, the fact that $|X_i| \neq |X_j|$ implies that $C^*(E_L(X_i)) $ is not Morita equivalent to $C^*(E_L(X_j))$.

For (3), the fact that the $C^*(E_K(X_i))$ are non-AF $C^*$-algebras of real rank zero that are prime and not primitive and that  all ideals are gauge invariant follows from Proposition~\ref{graphproperties}(3).   It remains to show the $C^*(E_K(X_i))$ are mutually non-Morita equivalent.  The proof follows much like the proof of part (1): For each $i \in I$, the graph $E_K(X_i)$ satisfies Condition~(K), and hence the ideals in $C^*(E_K(X_i))$ are in one-to-one correspondence with admissible pairs $(H,S)$, where $H$ is a saturated hereditary subset $E_K(X_i)^0$ and $S$ is a subset of breaking vertices for $H$. For each $x \in X_i$ define $H_x := E_K(X_i)^0 \setminus \{ x \}$.  Because $\{ x \}$ is a source and an infinite emitter in $E_A(X_i)$, the set $H_x$ is saturated and hereditary.  In addition, because there is a loop at $\{ x \}$, it is a (unique) breaking vertex for $H_x$.  In addition, any saturated hereditary subset of $E_K(X_i)^0$ that contains $H_x$ is either equal to $H_x$ or equal to $E_A(X_i)^0$.  Thus $I_{(H_x,\{ x \})}$ is a maximal ideal in $C^*(E_K(X_i))$.  Conversely, any maximal ideal must have the form $I_{(H,S)}$ for an admissible pair $(H,S)$, and in order to be a proper ideal there exists $x \in E_K(X_i)^0 \setminus H$.  Because $H \subseteq H_x$, the maximality of $I_{(H,S)}$ implies $H= H_x$ and $S = \{ x \}$.  Thus we may conclude that the map $x \mapsto I_{(H_x, \{ x \})}$ is a bijection from $X$ onto the set of maximal ideals of $C^*(E_K(X_i))$.  Thus the cardinality of the set of maximal ideals of $C^*(E_K(X_i))$ is equal to $| X_i|$, and as argued in (1) this implies that when $i,j \in I$ with $i \neq j$, the $C^*$-algebra $C^*(E_K(X_i)) $ is not Morita equivalent to $C^*$-algebra $C^*(E_K(X_j))$.
\end{proof}

\begin{remark}
Note that in each of parts (1)--(3) of Proposition~\ref{examplesofprimenotprimitive} we are constructing a prime, nonprimitive $C^*$-algebra for each set in the collection $\mathcal{C}$.  We mention that for any cardinal number $\kappa$, there exists a collection of $\kappa$ sets of differing cardinalities all greater than or equal to $2^{\aleph_0}$.   (This fact is  well known; see e.g. \cite[Lemma 7.7]{Jech}.)   Hence in each of parts (1)--(3) of Proposition~\ref{examplesofprimenotprimitive} one can choose the collection $\mathcal{C}$ to be of any desired cardinality $\kappa$.  \end{remark}

\begin{example}\label{uncountableseparableexamples}
There are, of course, many examples of uncountable graphs whose associated $C^*$-algebras are primitive.  For instance, let $X$ be any uncountable set,  let $\mathcal{P}(X)$ denote the set of {\it all} subsets of $X$, and let $E_\mathcal{P}(X)$ be the graph having 
$$E_{\mathcal{P}(X)}^0 =  \mathcal{P}(X) \quad \text{ and } \quad E_{\mathcal{P}(X)}^1 = \{e_{A,A'} \ | \ A,A'\in {\mathcal{P}(X)} \text{ and }  A\subsetneqq A'\},$$
with $s(e_{A,A'})=A$, and $r(e_{A,A'})=A'$ for each $e_{A,A'} \in E_{\mathcal{P}(X)}^1.$
Then $E_{\mathcal{P}(X)}$ is not a countable graph, and $C^*(E_{\mathcal{P}(X)})$ is not a separable $C^*$-algebra.   However, $E_{\mathcal{P}(X)}$ satisfies the three conditions of Theorem~\ref{primitivitytheorem}, and hence $C^*(E_{\mathcal{P}(X)})$ is a primitive $C^*$-algebra.  (In particular, we observe that any vertex in $E_{\mathcal{P}(X)}$ emits an edge pointing to $\{X\} \in E_{\mathcal{P}(X)}^0$, so $E_{\mathcal{P}(X)}$ trivially satisfies the Countable Separation Property.)  The graph $E_{\mathcal{P}(X)}$ has no cycles, so that $C^*(E_{\mathcal{P}(X)})$ is an AF-algebra.  In a like manner, we could construct additional examples of uncountable graphs having primitive graph $C^*$-algebras, and one could easily produce non-AF examples by adding one or two loops at every vertex of $E_{\mathcal{P}(X)}$.
\end{example}

The following definition provides  a second graph-theoretic construction which produces examples of graphs whose corresponding graph $C^*$-algebras are prime but not primitive.

\begin{definition}\label{Esubkappadefinition}
{\rm Let $\kappa > 0$ be any ordinal.  We  define the graph $E_{\kappa}$ as follows:
$$E_{\kappa}^0 \ = \{\alpha \ | \ \alpha < \kappa\}, \ \ \ E_{\kappa}^1 = \{e_{\alpha,\beta} \ | \ \alpha, \beta < \kappa, \ \mbox{and} \ \alpha < \beta\},$$
  $s(e_{\alpha,\beta})=\alpha$, and $r(e_{\alpha,\beta})=\beta$ for each $e_{\alpha,\beta} \in E_{\kappa}^1.$
 \hfill  $\Box$  }
\end{definition}

Recall that an ordinal $\kappa$ is said to have {\it countable cofinality} in case $\kappa$ is the limit of a countable sequence of ordinals strictly less than $\kappa$.   For example, any countable ordinal has countable cofinality.  The ordinal $\omega_1$ does not have countable cofinality, while the ordinal $\omega_\omega$ does have this property.   With this definition in mind, it is clear that  $E_{\kappa}^0$ has the Countable Separation Property if and only if $\kappa$ has countable cofinality.
Thus by Theorem~\ref{primitivitytheorem} we get 

\begin{proposition}\label{kappasubalphaexamples}
 Let $\{\kappa_\alpha \ | \ \alpha \in \mathcal{A}\}$ denote a set of distinct ordinals, each without  countable cofinality.  Then the collection $\{C^*(E_{\kappa_\alpha}) \ | \ \alpha \in \mathcal{A} \}$ is  a set  of nonisomorphic AF-algebras, each of which is prime but not primitive.  
\end{proposition}

\section{A comparison of the conditions for a graph $C^*$-algebra to be simple, to be primitive, and to be prime}

In Proposition~\ref{primenessproposition}  and Theorem~\ref{primitivitytheorem}, we obtained conditions for a graph $C^*$-algebra to be prime and primitive, respectively.  In this section we compare these conditions with the conditions for a graph $C^*$-algebra to be simple.   Recall that an {\it infinite path} $p$ {\it in}  $E$ is a nonterminating sequence $p = e_1e_2e_3 \dots$ of edges in $E$, for which $r(e_i) = s(e_{i+1})$ for all $i\geq 1$.   (Note that this notation, while standard, can be misleading; an infinite path in $E$ is not an element of ${\rm Path}(E)$, as the elements of ${\rm Path}(E)$ are, by definition,  finite sequences of edges in $E$.)  We denote the set of infinite paths in $E$ by $E^\infty$. 

\begin{proposition}\label{threesetsofconditions}
Let $E$ be a graph.

\noindent The graph $C^*$-algebra $C^*(E)$ is \textbf{simple} if and only if the following two conditions are satisfied
\begin{itemize}
\item[(1)] $E$ satisfies Condition~(L), and
\item[(2)] $E$ is cofinal (i.e., if $v \in E^0$ and $\alpha \in E^\infty \cup E^0_\textnormal{sing}$, then $v \geq \alpha^0$).
\end{itemize}

\noindent The graph $C^*$-algebra $C^*(E)$ is \textbf{primitive} if and only if the following three conditions are satisfied
\begin{itemize}
\item[(1)] $E$ satisfies Condition~(L),
\item[(2)] $E$ is downward directed (i.e., for all $v,w \in E^0$ there exists $x \in E^0$ such that $v \geq x$ and $w \geq x$), and 
\item[(3)] $E$ satisfies the Countable Separation Property (i.e., there exists a countable set $X \subseteq E^0$ such that $E^0 \geq X$).
\end{itemize}

\noindent The graph $C^*$-algebra $C^*(E)$ is \textbf{prime} if and only if the following two conditions are satisfied
\begin{itemize}
\item[(1)] $E$ satisfies Condition~(L), and
\item[(2)] $E$ is downward directed (i.e., for all $v,w \in E^0$ there exists $x \in E^0$ such that $v \geq x$ and $w \geq x$). 
\end{itemize}
\end{proposition}

\begin{proof}
Proposition~\ref{primenessproposition} and Theorem~\ref{primitivitytheorem} give the stated conditions for a graph $C^*$-algebra to be prime and primitive, respectively.  The conditions for simplicity are established in \cite[Theorem~2.12]{Tom9}.  (Although all of \cite{Tom9} is done under the implicit assumption the graphs are countable, the proof of \cite[Theorem~2.12]{Tom9} and the proofs of the results on which it relies all go through verbatim for uncountable graphs.)
\end{proof}

Every $C^*$-algebra has a nonzero irreducible representation. (This follows from the GNS construction, which shows that GNS-representations constructed from pure states are irreducible \cite[Lemma~A.12]{RW}, and the Krein-Milman Theorem, which asserts that pure states exist for any $C^*$-algebra \cite[Lemma~A.13]{RW}.)  Thus any simple $C^*$-algebra has a faithful irreducible $*$-representation, and any simple $C^*$-algebra is necessarily primitive.   Moreover, as was shown in Lemma~\ref{primitiveimpliesprime}, any primitive $C^*$-algebra is necessarily prime.  Thus we have 
 $$ \text{$C^*(E)$ is simple} \  \implies \ \text{$C^*(E)$ is primitive} \ \implies \ \text{$C^*(E)$ is prime.} $$

We observe that the form of each of the three  results presented in Proposition \ref{threesetsofconditions}, in which simplicity, primeness, and primitivity of $C^*(E)$ are given in graph-theoretic terms,   may be described as ``Condition (L) plus something extra".  This having been said, our goal for this section is solely graph-theoretic:  we  show that these three  ``extra" conditions may be seen as arising in a common  context, by considering subsets of $E^0$ of the form $\overline{H(v)}$.

\begin{proposition}\label{graphconditionsforthreeproperties}
Let $E$ be a graph.  Then the following equivalences hold.
\begin{itemize}
\item[(1)] $E$ is cofinal if and only if for all $v \in E^0$ one has $\overline{H(v)} = E^0$.
\item[(2)] $E$ is downward directed if and only if for all $v, w \in E^0$ one has $\overline{H(v)}  \cap \overline{H(w)}  \neq \emptyset$.
\item[(3)] $E$ satisfies the Countable Separation Property if and only if there exists a countable collection of subsets of vertices $\{ S_i : i \in I \}$ (so, $I$ is countable and $S_i \subseteq E^0$ for all $i \in I$) with $E^0 = \bigcup_{i \in I} S_i$ and with $\bigcap_{v \in S_i} \overline{H(v)}  \neq \emptyset$ for all $i \in I$.
\end{itemize}
\end{proposition}

\begin{proof}
For (1), suppose first that for all $v \in E^0$ we have $\overline{H(v)}  = E^0$.  Choose $v \in E^0$ and $\alpha \in E^\infty \cup E^0_\textnormal{sing}$.    If $\alpha \in E^0_\textnormal{sing}$, then $\alpha \in\overline{H(v)}  = E^0$.  Since every element of $\overline{H(v)}  \setminus H(v)$ is a regular vertex, it follows that $\alpha \in H(v)$ and $v \geq \alpha$.  If instead $\alpha \in E^\infty$, then we may write $\alpha = e_1 e_2 e_3 \ldots$ for $e_i \in E^1$ with $r(e_i) = s(e_{i+1})$ for all $i \in \N$.  Since $s(e_1) \in \overline{H(v)} = E^0$, it follows that $s(e_i) \in H(v)_n$ for some $n \in \N$  (recall the notation of Lemma \ref{sat-int-lem}).  Thus $s(e_2) = r(e_1) \in H(v)_{n-1}$,
and continuing recursively we have $s(e_n) = r(e_{n-1}) \in H(v)_1$, and $s(e_{n+1}) = r(e_n) \in H(v)_0 = H(v)$.  Hence $v \geq s(e_{n+1})$, and $v \geq \alpha^0$.  Thus $E$ is cofinal.  Conversely, if there exists $v \in E^0$ with $ \overline{H(v)} \neq E^0$, then there is a vertex $w \in E^0 \setminus \overline{H(v)}$.  Since $\overline{H(v)}$ is saturated, either $w \in E^0_\textnormal{sing}$ or there exists an edge $e_1 \in E^1$ with $s(e_1) = w$ and $r(e_1) \notin \overline{H(v)}$.  Using $r(e_1)$ and continuing inductively, we either produce a singular vertex $z \in E^0 \setminus \overline{H(v)}$ or an infinite path $\alpha := e_1 e_2 e_3 \ldots$ with $s(e_i) \in E^0 \setminus \overline{H(v)}$.  Since $H(v) \subseteq \overline{H(v)}$, it follows that either there is a singular vertex $z$ with $v \not\geq z$ or there is an infinite path $\alpha$ with $v \not\geq \alpha^0$.  Hence $E$ is not cofinal.

For (2), suppose first that $E$ is downward directed.  If $v,w \in E^0$, then the definition of downward directed implies that $H(v) \cap H(w) \neq \emptyset$.  Since $H(v) \subseteq \overline{H(v)}$ and $H(w) \subseteq \overline{H(w)}$, it follows that $\overline{H(v)} \cap \overline{H(w)} \neq \emptyset$.  Conversely, suppose that for all $v, w \in E^0$ one has $\overline{H(v)} \cap \overline{H(w)} \neq \emptyset$.  Then for any $v, w \in E^0$, we may choose $x \in \overline{H(v)} \cap \overline{H(w)}$.  Using the notation of Definition~\ref{sat-hered-def} write $\overline{H(v)} = \bigcup_{n=0}^\infty H(v)_n$ and $\overline{H(w)} = \bigcup_{n=0}^\infty H(w)_n$.  Choose the smallest $n_1 \in \N \cup \{ 0 \}$ such that $x \in  H(v)_{n_1}$, and choose the smallest $n_2 \in \N \cup \{ 0 \}$ such that $x \in  H(w)_{n_2}$.  If we let $n := \max \{ n_1, n_2 \}$, then $x \in H(v)_n \cap  H(w)_n$.  If $n= 0$, then $x \in H(v) \cap H(w)$ and we have that $v \geq x$ and $w \geq x$.  If $n \geq 1$, then there exists an edge $e_1 \in E^1$ such that $s(e_1) = x$ and $r(e_1) \in H(v)_{n-1} \cap H(w)_{n-1}$.  Using $r(e_1)$ next, and continuing recursively, we produce a finite path $\alpha := e_1 \ldots e_n$ with $r(e_n) \in H(v) \cap H(w)$.  Hence $v \geq r(e_n)$ and $w \geq r(e_n)$.  Thus $E$ is downward directed.

For (3), suppose first that $E$ satisfies the Countable Separation Property.  Then there is a countable nonempty set $X \subseteq E^0$ such that $E^0 = \bigcup_{x \in X} U(x)$.  In addition, $x \in \bigcap_{v \in U(x)} \overline{H(v)}$ for all $x \in X$, so $\bigcap_{v \in U(x)} \overline{H(v)} \neq \emptyset$ for all $x \in X$.  Thus the condition in (3) holds with $I := X$ and $S_i := U(i)$ for all $i \in I$.  Conversely, suppose that there exists a countable collection of subsets of vertices $\{ S_i : i \in I \}$ with $E^0 = \bigcup_{i \in I} S_i$ and with $\bigcap_{v \in S_i} \overline{H(v)}\neq \emptyset$ for all $i \in I$.  For each $i \in I$, choose a vertex $v_i \in \bigcap_{v \in S_i} \overline{H(v)}$ and define $$C_{v_i} := \{ r(\alpha) : \text{$\alpha \in \path (E)$, $s(\alpha) = v$, and $s(\alpha_i) \in E^0_\textnormal{reg}$ for all $1 \leq i \leq |\alpha|$} \}.$$  (Note that if $v_i \in E^0_\textnormal{sing}$, then $C_{v_i} := \{ v_i \}$.)  Since there are only a finite number of edges emitted from any regular vertex, we see that $C_{v_i}$ is a countable set for all $i \in I$.  We define $X := \bigcup_{i \in I} C_{v_i}$, and observe that since $X$ is a countable union of countable sets, $X$ is countable.  If $w \in E^0$ is any vertex, then by the hypothesis that $E^0 = \bigcup_{i \in I} S_i$ there exists $i \in I$ such that $w \in S_i$.  By the definition of $v_i$ we then have that $v_i \in \overline{H(w)}$.  Hence there exists a path $\alpha \in \path (E)$ such that $s(\alpha) = v_i$, $r(\alpha) \in H(w)$, and $s(\alpha_i) \in E^0_\textnormal{reg}$ for all $1 \leq i \leq |\alpha|$.  Thus $r(\alpha) \in C_{v_i}$, and $w \geq r(\alpha)$, so that $w \in \bigcup_{x \in X} U(x)$.  We have therefore shown that $E^0 = \bigcup_{x \in X} U(x)$, and hence $E$ satisfies the Countable Separation Property.
\end{proof}

It is clear that Property (1) of Proposition \ref{graphconditionsforthreeproperties} implies both Property (2) and Property (3) of that Proposition.  (Note that we may use the singleton set $S = E^0$ to establish Property (3) from Property (1).)   Thus, as promised, using Proposition \ref{threesetsofconditions}, we have established a natural connection between simplicity, primitivity, and primeness for graph $C^*$-algebras from a graph-theoretic point of view.   We summarize this observation as the following result. 

\begin{corollary}
Let $E$ be a graph.

\noindent The graph $C^*$-algebra $C^*(E)$ is \textbf{simple} if and only if the following two conditions are satisfied
\begin{itemize}
\item[(1)] $E$ satisfies Condition~(L)
\item[(2)] If $v \in E^0$, then $\overline{H(v)} = E^0$.
\end{itemize}

\noindent The graph $C^*$-algebra $C^*(E)$ is \textbf{primitive} if and only if the following three conditions are satisfied
\begin{itemize}
\item[(1)] $E$ satisfies Condition~(L)
\item[(2)] If $v, w \in E^0$, then $\overline{H(v)} \cap \overline{H(w)} \neq \emptyset$.
\item[(3)] There exists a countable collection of subsets of vertices $\{ S_i : i \in I \}$ (so, $I$ is countable and $S_i \subseteq E^0$ for all $i \in I$) such that $E^0 = \bigcup_{i \in I} S_i$ and $\bigcap_{v \in S_i} \overline{H(v)} \neq \emptyset$ for all $i \in I$.
\end{itemize}

\noindent The graph $C^*$-algebra $C^*(E)$ is \textbf{prime} if and only if the following two conditions are satisfied
\begin{itemize}
\item[(1)] $E$ satisfies Condition~(L)
\item[(2)] If $v, w \in E^0$, then $\overline{H(v)}\cap \overline{H(w)} \neq \emptyset$.
\end{itemize}
\end{corollary}

We conclude this graph-theoretic section with the following observation.  Since the simplicity of $C^*(E)$ clearly implies its primeness, it is perhaps not surprising that there should be a direct graph-theoretic connection between the germane properties appearing in Proposition \ref{threesetsofconditions}.   Indeed, one can easily see that cofinal implies downward directed, as follows. If $E$ is cofinal, and $v, w \in E^0$, then one may inductively create a sequence of edges $\alpha := e_1 e_2 e_3 \ldots$ with $s(e_1) = w$ and $s(e_i) = r(e_{i-1})$ for all $i \geq 2$, and such that this sequence either ends at a sink or goes on forever to produce an infinite path.  Hence either $v \geq r(\alpha)$ (if $\alpha$ ends at a sink) or $v \geq \alpha^0$ (if $\alpha$ is an infinite path), and $E$ is downward directed.  

From this point of view, the difference between the notion of cofinal and the notion of downward directed can be viewed as follows:  $E$ is cofinal if and only if ``for all $v,w \in E^0$ and for {\it  all} $\alpha \in \path (E)$ with $s(\alpha) = w$ there exists $x \in E^0$ such that $v \geq x$ and $r(\alpha) \geq x$", while $E$ is downward directed if and only if ``for all $v,w \in E^0$ and for {\it some} $\alpha \in {\rm Path}(E)$ with $s(\alpha) = w$  there exists $x \in E^0$ such that $v \geq x$ and $r(\alpha) \geq x$."  Specifically, the cofinality property allows for the path from one of the vertices to start along any specified initial segment $\alpha$, while the downward directedness property contains no such requirement.

\section{Primality and primitivity of graph $C^*$-algebras compared with primality and primitivity of Leavitt path algebras}

In this final section we  compare the notions of primeness and primitivity for graph $C^*$-algebras $C^*(E)$ with primeness and primitivity for Leavitt path algebras.  Briefly, for any graph $E$ and any field $K$, one may define the $K$-algebra $L_K(E)$, the  {\it Leavitt path algebra of} $E$ {\it with coefficients in} $K$. When $K = \C$, then $L_\C(E)$ may be viewed as a dense $*$-subalgebra of $C^*(E)$.   For reasons which remain not well understood, many structural properties are simultaneously shared by both $L_\C(E)$ and $C^*(E)$.   We show in this section that the primitivity property may be added to this list.   Additional information about Leavitt path algebras may be found in  \cite{AAS} or   \cite{ABR}.

The map $\sum_{i=1}^n \lambda _i \alpha_i \beta_i^* \mapsto \sum_{i=1}^n \lambda _i \beta_i \alpha_i^*$ is a $K$-algebra isomorphism from $L_K(E)$ onto its opposite algebra $L_K(E)^\textnormal{op}$.  Hence there is a natural correspondence between left $L_K(E)$-modules and right $L_K(E)$-modules, which yields 

\begin{proposition} \cite[Proposition 2.2]{ABR}  \label{LPA-left-prim-iff-right-prim-lem}
If $E$ is a graph and $K$ is a field, then the algebra $L_K(E)$ is left primitive if and only if it is right primitive.
\end{proposition}

\begin{definition}
In light of Proposition~\ref{LPA-left-prim-iff-right-prim-lem}, we shall say a Leavitt path algebra $L_K(E)$ is \emph{primitive} if it is left primitive (which occurs if and only if $L_K(E)$ is also right primitive).
\end{definition}

\begin{remark}
When we say $L_K(E)$ is primitive, we mean that $L_K(E)$ is primitive as a ring.  The astute reader may notice that it seems more natural to consider primitivity of $L_K(E)$ as an algebra; that is, to reformulate the definition of primitive as having a simple faithful left \emph{$K$-algebra} module (not merely a simple faithful left \emph{ring} module).  However, since $L_K(E)$ has local units, any ring module also carries a natural structure as a $K$-algebra module.  Hence any Leavitt path algebra is primitive as a ring if and only if it is primitive as a $K$-algebra.  Likewise, again using that  Leavitt path algebras have local units, any ring ideal of $L_K(E)$ is closed under scalar multiplication by $K$, and hence the ring ideals of $L_K(E)$ are precisely the $K$-algebra ideals of $L_K(E)$.  Consequently, a Leavitt path algebra is prime as a ring if and only if it is prime as a $K$-algebra, and a Leavitt path algebra is simple as a ring if and only if it is simple as a $K$-algebra.   Thus for Leavitt path algebras the ring-theoretic notions of primitive, prime, and simple coincide with the corresponding $K$-algebra-theoretic notions.
\end{remark}

\begin{theorem}\label{CstarEprimitiveiffL(E)primitive}
Let $E$ be a graph.  Then the following are equivalent.
\begin{itemize}
\item[(i)]   $C^*(E)$ is primitive.
\item[(ii)]  $L_K(E)$ is primitive for some field $K$.
\item[(iii)]  $L_K(E)$ is primitive for every field $K$.  
\item[(iv)]  $E$ satisfies Condition~(L), is downward directed, and satisfies the Countable Separation Property. 
\end{itemize} 
\end{theorem}

\begin{proof}
The equivalence of (i) and (iv) is precisely Theorem~\ref{primitivitytheorem}.  The equivalence of (ii), (iii), and (iv), is shown in \cite[Theorem 5.7]{ABR}. 
\end{proof}

\medskip

Theorem~\ref{CstarEprimitiveiffL(E)primitive} provides yet another example of a situation in which the same ring-theoretic property holds for both of the algebras $C^*(E)$ and $L_\C(E)$ (indeed, $L_K(E)$ for any field $K$), but for which the proof  that the pertinent property holds in each case is wildly different.  In particular, no ``direct" connection between $C^*(E)$ and $L_\C(E)$ is established.  We note that the proof of the sufficiency direction of Theorem~\ref{primitivitytheorem} looks, on the surface, nearly identical to the proof that $L_\C(E)$ is primitive whenever $E$  satisfies Condition~(L), is downward directed, and satisfies the Countable Separation Property  \cite[Theorem 3.5 with Proposition 4.8]{ABR}.    However, in the proof of the result herein we invoke the Cuntz-Krieger Uniqueness Theorem, whose justification is significantly different than that of the correspondingly invoked algebraic result \cite[Corollary 3.3]{ABGM}.  Furthermore,  the proof of the necessity direction of Theorem~\ref{primitivitytheorem}  is significantly different than the proof of the analogous result for Leavitt path algebras \cite[Proposition 5.6]{ABR}.  In this regard, it is worth noting that for Leavitt path algebras, in contrast to Lemma~\ref{nouncountablesetsoforthogidempotents} for $C^*$-algebras, it is perfectly possible to have a graph $E$ and left $L_K(E)$-module $M$ containing an element  $m$  for which there exists an uncountable set of nonzero orthogonal projections in $L_K(E)$ which do not annihilate $m$.   
For example, let $U$ be an uncountable set, and let $E_U$ denote the graph having one vertex $v$, and uncountably many loops $\{e_i \ \vert \ i\in U\}$ at $v$.   Let    $R = L_K(E_U)$, and let $M = {}_RR$.  Then for $m = 1_R\in M$, $\{e_ie_i^* \ \vert \ i\in U\}$ is such a set.

In contrast to the result presented in Theorem \ref{CstarEprimitiveiffL(E)primitive}, the class of graphs which produce prime Leavitt path algebras is not the same as the class of graphs which produce prime graph $C^*$-algebras.   For example, if we let $E$ be the graph with one vertex and one edge

$ $

$$
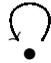
\xymatrix{
\bullet \ar@(ur,ul) \\
}
$$
then for any field $K$, the Leavitt path algebra $L_K(E)$ is isomorphic to $K[x,x^{-1}]$, the algebra of Laurent polynomials with coefficients in $K$, which is prime.  (Indeed, $K[x,x^{-1}]$ is a commutative integral domain.)   However, the graph $C^*$-algebra $C^*(E)$ is isomorphic to $C(\T)$, the $C^*$-algebra of continuous functions on the circle, which is not prime.  

Thus ``primeness"  yields  one of the relatively uncommon contexts in which  an algebraic property of $L_K(E)$ does not coincide with the corresponding $C^*$-algebraic property of $C^*(E)$.  Hence the conditions on $E$ for $L_K(E)$ to be prime are different than the conditions on $E$ for $C^*(E)$ to be prime.

Necessary and sufficient conditions for a Leavitt path algebra to be prime are given in \cite[Corollary~3.10]{APS06} (see also   \cite[Theorem 2.4]{ABR}), which we state here.

\begin{proposition}\label{LPA-prime-iff-prop}
Let $E$ be a graph.  Then the following are equivalent
\begin{itemize}
\item[(i)] $L_K(E)$ is prime for some field $K$.
\item[(ii)] $L_K(E)$ is prime for every field $K$.
\item[(iii)] $E$ is downward directed. 
\end{itemize}
\end{proposition}

We conclude this article with the following summary of comparisons of germane properties between Leavitt path algebras and graph $C^*$-algebras.  A proof that the indicated conditions on $E$ which yield the simplicity of $L_K(E)$ for any field $K$ is given in \cite{AAS}.  The remaining comparisons follow from   Proposition~\ref{LPA-prime-iff-prop} with Proposition~\ref{primenessproposition} and Theorem~\ref{CstarEprimitiveiffL(E)primitive}.

\bigskip
\bigskip

$\begin{array}{c} \text{$L_K(E)$ is} \\ \text{simple} \end{array}$ \hspace{-.15in}
$\iff$ \hspace{-.1in} $\begin{array}{c} \text{$C^*(E)$ is} \\ \text{simple} \end{array}$ \hspace{-.15in} $\iff$$\begin{array}{l} \text{$E$ is cofinal, and}  \\ \text{$E$ satisfies Condition (L)}\end{array}$

\bigskip
\bigskip

\begin{center}
\hspace{-0.63in} $L_K(E)$ prime $\iff$ $E$ is downward directed
\end{center}

\medskip

\begin{center}
$C^*(E)$ prime $\iff$ $\begin{array}{l} \text{$E$ is downward directed, and} \\ \text{$E$ satisfies Condition~(L)} \end{array}$
\end{center}

$ $

\bigskip

$ $

$\begin{array}{c} \text{$L_K(E)$ is} \\ \text{primitive} \end{array}$ \hspace{-.15in}
$\iff$ \hspace{-.1in} $\begin{array}{c} \text{$C^*(E)$ is} \\ \text{primitive} \end{array}$ \hspace{-.15in} $\iff$$\begin{array}{l} \text{$E$ is downward directed,} \\ \text{$E$ satisfies Condition~(L), and} \\ \text{$E$ has the Countable Separation Property}\end{array}$

$ $

$ $

\noindent In particular, we note that $C^*(E)$ is prime if and only if $L_K(E)$ is prime and $E$ satisfies Condition~(L).  Specifically, $C^*(E)$ prime implies $L_K(E)$ prime for every field $K$, but not conversely.

\end{document}